\documentclass{article}

\usepackage{amsmath,amssymb}
\usepackage{amsthm}
\usepackage[pdftex,bookmarks=true]{hyperref}
\usepackage{cite}
\usepackage[capitalise, noabbrev, nameinlink]{cleveref}

\usepackage{tikz}
\tikzset{
every node/.style={draw, circle, inner sep=2pt}
}
\usetikzlibrary{arrows}

\usepackage{soul}
\usepackage{cancel}

\newtheorem{theorem}{Theorem}[section]

\newtheorem{proposition}[theorem]{Proposition}
\newtheorem{corollary}[theorem]{Corollary}

\theoremstyle{definition}
\newtheorem{definition}[theorem]{Definition}

\newtheorem{remark}[theorem]{Remark}
\newtheorem{example}[theorem]{Example}

\newcommand{\trans}{^\top}

\newcommand{\bzero}{\mathbf{0}}

\newcommand{\bx}{\mathbf{x}}

\newcommand{\eval}{\operatorname{eval}}
\newcommand{\fchar}{\operatorname{char}}

\title{Confluent Vandermonde matrix and related topics}
\author{
Chi-Kwong Li
\thanks{Department of Mathematics, College of William \& Mary, Williamsburg, VA 23187, USA. (ckli@math.wm.edu)}
\and
Jephian C.-H.~Lin
\thanks{Department of Applied Mathematics, National Sun Yat-sen University, Kaohsiung 80424, Taiwan (jephianlin@gmail.com)}
}
\date{\today}

\begin{document}

\maketitle

\begin{abstract}
In this note, we explore the connections between the confluent Vandermonde matrix over an arbitrary field and several mathematical topics, including 
interpolation polynomials, Hasse derivatives, LU factorization, companion matrices and their Jordan forms, and the partial fraction decomposition. Using a unified approach based on polynomial evaluations and derivative computations at selected points, we 
provide accessible proofs that not only clarify key results but also offer insights for both experienced researchers and those new 
to the subject.
\end{abstract}  


\medskip



\section{Introduction}

Let $\Lambda = \{\lambda_1, \ldots, \lambda_d\}$ be a set of $d$ distinct real numbers.  The corresponding \emph{Vandermonde matrix} is the $d\times d$ matrix
\begin{equation}
\label{eq:vlam}
    V_\Lambda = \begin{bmatrix}
      1 & \lambda_1 & \cdots & \lambda_1^{d-1} \\
      1 & \lambda_2 & \cdots & \lambda_2^{d-1} \\
      \vdots & \vdots & ~ & \vdots \\
      1 & \lambda_d & \cdots & \lambda_d^{d-1} \\
    \end{bmatrix}.
\end{equation}
The Vandermonde matrix has been used in polynomial interpolation, coding theory, signal processing, differential equations, and recursive relations; see \cite{Rushanan89, Kalman84, Eller87, Brand68, MacDuffee50} and the references therein.  It was also used
to give alternative proofs for some classical results. For example, it was used to
prove a theorem of Abel \cite{GT93}, whose classical proof is based on the residue theorem.  
There are many elegant results on the Vandermonde matrix, and one of them is the following formula for the determinant of $V_\Lambda$ which appeared in many textbooks
\begin{equation}
\label{eq:det}
    \det(V_\Lambda) = \prod_{i<j}  (\lambda_j - \lambda_i).
\end{equation}
Rushanan\cite{Rushanan89} gave the LU factorization of $V_\Lambda$.
Many formulas for $V_\Lambda^{-1}$ were studied and rediscovered by researchers in pure and applied subjects; 
see, e.g., \cite{MS58, Klinger67}.  In particular, 
Parker\cite{Parker64} pointed out that the columns of $V_\Lambda^{-1}$ are exactly the coefficients of the Lagrange interpolation polynomials.  The generalization of $V_\Lambda$ whose columns are assigned with nonconsecutive powers was considered and revealed a strong
connection to symmetric functions through Schur functions; see, e.g., \cite{Mitchell1881, Heineman29,  Goodstein70, EI76, dSilva18, 
EPY24, Aitken39}.

Using an example, Aitken and Turnbull \cite[Chapter VI]{AT32} demonstrated an interesting fact that $V_\Lambda C_\Lambda V_\Lambda^{-1}$ is a diagonal matrix, where $C_\Lambda$ 
is the companion matrix of $(x - \lambda_1) \cdots (x - \lambda_d)$.  In the same chapter, the authors mentioned that
if $\lambda_1, \dots, \lambda_d$ are not distinct, 
one can extend the definition of the Vandermonde matrix to define the confluent Vandermonde matrix $V_\Lambda$ using the
higher derivatives of the polynomial $(x-\lambda_1)\cdots (x-\lambda_d)$ at the points $\lambda_j$,
which appear more than once; see the definition in \cref{sec:eval}. Then
$V_\Lambda C_\Lambda V_\Lambda^{-1}$ will be the Jordan canonical form of $C_\Lambda$.  The same points were later emphasized in Brand \cite{Brand64}.  

Researchers have obtained results on the confluent Vandermonde matrix because of its connections
and applications to many different subjects. However, these results are scattered across different areas of the
mathematical literature. In this paper, we give a systematic introduction to the confluent Vandermonde matrix. In particular,  
using a unified approach based on polynomial evaluations and derivative computations at selected points, we
provide accessible proofs that not only clarify key results but also offer insights for both experienced researchers and those new
to the subject.

\section{Evaluations of a polynomial at specific points}
\label{sec:eval}

Let  
\begin{equation}
\label{eq:poly}
    p(x) = c_0 + c_1x + \cdots + c_{d-1}x^{d-1}
\end{equation}
be a polynomial over $\mathbb{R}$.  Let $\Lambda = \{\lambda_1, \ldots, \lambda_d\}$ be a set of $d$ 
distinct real numbers and $V_\Lambda$ be as in \cref{eq:vlam}.  Then  
\begin{equation}
\label{eq:veval}
    V_\Lambda
    \begin{bmatrix} c_0 \\ \vdots \\ c_{d-1} \end{bmatrix}
    =
    \begin{bmatrix} p(\lambda_1) \\ \vdots \\ p(\lambda_d) \end{bmatrix}.
\end{equation}
Therefore, $V_\Lambda$ can be viewed as a transformation from a polynomial of degree at most $d-1$
to its evaluations at the points in $\Lambda$.  In fact, $V_\Lambda$ is a linear transformation from the 
linear space of polynomials with degrees bounded by $d-1$ to ${\mathbb R}^d$. 
Since $\det(V_\Lambda) \neq 0$, $V_\Lambda$ is invertible and any assignment of 
$p(\lambda_1), \ldots, p(\lambda_d)$ defines a unique polynomial $p(x)$ of degree bounded by $d-1$.  

When $\Lambda$ contains repeated values, the repeated $\lambda_i$ should provide not just the 
information $p(\lambda_i)$.  Instead, more information should be revealed at the point of $\lambda_i$, and 
the evaluations at its derivatives $p(\lambda_i)$, $p'(\lambda_i)$, $\ldots$, are natural choices.
In fact, one may view the higher derivatives at $\lambda_i$ as the information obtained by the limit process
when several distinct points converge to $\lambda_i$.

Note that if $p(x) = c_0 + c_1 x + \cdots + c_{d-1} x^{d-1}$, then
\[
    \begin{array}{lcrrrr}
        \frac{1}{0!}\cdot p(x) &= & c_0 +& c_1x +& c_2x^2 +& \cdots + c_{d-1}x^{d-1}, \\
        \frac{1}{1!}\cdot\frac{d}{dx}p(x) &= & & c_1 +& c_22x +& \cdots + c_{d-1}(d-1)x^{d-2}, \\
        \frac{1}{2!}\cdot\frac{d^2}{dx^2}p(x) &= & & & c_2 +& \cdots + c_{d-1}\left(\frac{(d-1)(d-2)}{2}\right)x^{d-3}, \\
        & \vdots &&&&
    \end{array}
\]
and, in general,
\[
    \frac{1}{r!}\cdot \frac{d^r}{dx^r} p(x) = 
    \sum_{i = r}^{d-1} c_i\binom{i}{r}x^{i-r}.
\]
Let $\mathcal{D}^{(r)} = \frac{1}{r!}\cdot \frac{d^r}{dx^r}$.  Then for any  $m \le d$,
\begin{equation}
\label{eqn:solve-der}
    \begin{bmatrix}
        1 & \lambda & \lambda^2 & ~ & \cdots & \lambda^{d-1} \\
        0 & \binom{1}{1} & \binom{2}{1}\lambda & ~ & \cdots & \binom{d}{1}\lambda^{d-2} \\
        0 & 0 & \binom{2}{2} & ~ & \cdots & \binom{d}{2}\lambda^{d-3} \\
        ~ & ~ & ~ & ~ & \vdots & ~ \\
        0 & \cdots & 0 & \binom{m-1}{m-1} & \cdots & \binom{d}{m-1}\lambda^{d-m}
    \end{bmatrix}
    \begin{bmatrix} c_0 \\ \vdots \\ c_{d-1} \end{bmatrix} 
    = 
    \begin{bmatrix} \mathcal{D}^{(0)}p(\lambda) \\ \vdots \\ \mathcal{D}^{(m-1)}p(\lambda) \end{bmatrix}.
\end{equation}

Kalman \cite{Kalman84} gave a detailed discussion on this approach, while the confluent Vandermonde matrix to be defined below can be traced back to Aitken and Turnbull \cite{AT32} and Aitken \cite{Aitken39}.

\begin{definition}
\label{def:vlamd}
Let $\lambda$ be a real number and $m,d$ integers with $m \leq d$.  Define the $m\times d$ matrix 
\[
    V_{\lambda^{(m)}}^{(d)} = 
    \begin{bmatrix}
        1 & \lambda & \lambda^2 & ~ & \cdots & \lambda^{d-1} \\
        0 & \binom{1}{1} & \binom{2}{1}\lambda & ~ & \cdots & \binom{d}{1}\lambda^{d-2} \\
        0 & 0 & \binom{2}{2} & ~ & \cdots & \binom{d}{2}\lambda^{d-3} \\
        ~ & ~ & ~ & ~ & \vdots & ~ \\
        0 & \cdots & 0 & \binom{m-1}{m-1} & \cdots & \binom{d}{m-1}\lambda^{d-m}
    \end{bmatrix}.
\]
Let $\Lambda = \{\lambda_1^{(m_1)}, \ldots, \lambda_q^{(m_q)}\}$ be a multiset consisting of distinct
numbers $\lambda_1, \dots, \lambda_q$ with multiplicities $m_1, \dots, m_q$ such that
$m_1 + \cdots + m_q = d$.  Define 
\[
    V_\Lambda = 
    \begin{bmatrix}
        - & V_{\lambda_1^{(m_1)}}^{(d)} & - \\
         & \vdots & \\
        - & V_{\lambda_q^{(m_q)}}^{(d)} & - \\
    \end{bmatrix}.
\]
The square matrix $V_\Lambda$ is called 
the \emph{confluent Vandermonde matrix} (a.k.a.\ \emph{the generalized Vandermonde matrix})
of $\Lambda$.  
\end{definition}

\begin{example} \label{basic-example}
Suppose $p(x) = c_0+c_1x + \cdots + c_4x^4$ and 
$\Lambda = \{\lambda_1^{(1)}, \lambda_2^{(2)}, \lambda_3^{(2)}\}$. Then 
$$V_\Lambda = \begin{bmatrix}
1 & \lambda_1 & \lambda_1^2 & \lambda_1^3& \lambda_1^{4} \\
1 & \lambda_2 & \lambda_2^2 & \lambda_2^3& \lambda_2^{4} \\
0 & 1 & 2 \lambda_2 & 3\lambda_2^2 & 4\lambda_2^3  \\
1 & \lambda_3 & \lambda_3^2 & \lambda_3^3 &  \lambda_3^{4} \\
0 & 1 & 2 \lambda_3 & 3\lambda_3^2 & 4\lambda_3^3  \\
\end{bmatrix}.$$
\end{example}

The confluent Vandermonde matrix $V_\Lambda$ of 
$\Lambda = \{\lambda_1^{(m_1)}, \ldots, \lambda_q^{(m_q)}\}$ with $m_1 + \cdots + m_q = d$  
has the property that 
\begin{equation}
\label{eq:vgeval}
    V_\Lambda 
    \begin{bmatrix} c_0 \\ \vdots \\ c_{d-1} \end{bmatrix} = 
    \begin{bmatrix}
    \mathcal{D}^{(0)}p(\lambda_1) \\ \vdots \\ \mathcal{D}^{(m_1-1)}p(\lambda_1) \\ 
    \vdots \\
    \mathcal{D}^{(0)}p(\lambda_q) \\ \vdots \\ \mathcal{D}^{(m_q-1)}p(\lambda_q)
    \end{bmatrix}.
\end{equation}
Let $[p(x)] = [c_0, \dots, c_{d-1}]\trans$ 
be the vector of coefficients of $p(x) = c_0 + c_1x + \cdots + c_{d-1}x^{d-1}$ 
and $\eval_\Lambda(p(x))$ the the evaluations of $p(x)$ 
as shown in the right-hand side of \cref{eq:vgeval}.  
Hence, we have $V_\Lambda [p(x)] = \eval_\Lambda(p(x))$ extending the classical result \cref{eq:veval}.

In summary, $V_\Lambda$ is a matrix that transforms a polynomial into its evaluations at designated points
and their derivatives according to the multiplicities.  

\section{Hasse derivative and confluent Vandermonde matrices over arbitrary fields}
 
Let ${\mathbb F}[x]$ be the set of polynomials over 
a general field $\mathbb{F}$, and let 
$p(x) = c_0 + c_1x + \cdots + c_{d-1}x^d\in {\mathbb F}[x]$.
For any nonnegative integer $r$, define the \emph{Hasse derivative} of $p(x)$ by
\[
    \mathcal{D}^{(r)} p(x) = 
    \sum_{i = r}^{d-1} c_i\binom{i}{r}x^{i-r}.
\]  
When $\mathbb{F} = \mathbb{R}$, it is the same as the definition $\mathcal{D}^{(r)} = \frac{1}{r!}\cdot \frac{d^r}{dx^r}$ 
mentioned in the last section. 

Note that the binomial coefficient $\binom{n}{k}$ is the coefficient of $x^k$ in the expansion of $(1 + x)^n$, which is well 
defined in any field and is the same as $\frac{n!}{k!(n-k)!}$ when $\mathbb{F}$ is of characteristic $0$.  For this sake, an 
alternative definition of the Hasse derivative is $\mathcal{D}^{(r)}p(x) = [y^r]p(x + y)$, which is the coefficient of $y^r$ in 
$p(x + y)$ when treated as a polynomial of $y$.  This can be verified by direct computation as follows.
\[
    \begin{aligned}\relax
        [y^r]p(x + y) &= [y^r]\sum_{i=0}^{d-1} c_i(x+y)^i
        = \sum_{i=0}^{d-1} c_i \cdot [y^r] (x+y)^i \\
        &= \sum_{i=r}^{d-1} c_i \cdot [y^r] (x+y)^i 
        = \sum_{i=r}^{d-1} c_i \cdot \binom{i}{r}x^{i-r} = \mathcal{D}^{(r)}p(x)
    \end{aligned}
\]

More intuitively, the evaluation $\mathcal{D}^{(r)}p(\lambda)$ is exactly the coefficient of $(x - \lambda)^r$ in the 
``Taylor's expansion'' of $p(x)$ at $\lambda$ as shown in the following. 

\begin{proposition}
\label{prop:taylor}
Let $\lambda\in\mathbb{F}$.  When $p(x) = c_0 + c_1x + \cdots + c_{d-1}x^{d-1}$ is written as 
\[
    p(x) = a_0 + a_1(x - \lambda) + \cdots + a_{d-1}(x - \lambda)^{d-1}.
\]
Then 
\[
    \mathcal{D}^{(r)}p(x) = \sum_{i=r}^{d-1} a_i \binom{i}{r}(x - \lambda)^{i-r}.
\]
Therefore, $\mathcal{D}^{(r)}p(\lambda) = a_r$ can be viewed as the $r$-th order term of $p(x)$ in its Taylor's expansion at $\lambda$.  
\end{proposition}
\begin{proof}
With the alternative definition, we may compute
\[
    \begin{aligned}
        \mathcal{D}^{(r)}p(x) &= [y^r]p(x + y) = [y^r] \sum_{i=0}^{d-1} a_i (x + y - \lambda)^i \\
        &= \sum_{i=0}^{d-1} a_i [y^r] (x - \lambda + y)^i 
        = \sum_{i=r}^{d-1} a_i [y^r] (x - \lambda + y)^i \\
        &= \sum_{i=r}^{d-1} a_i \binom{i}{r} (x - \lambda)^{i-r}. \\
    \end{aligned}
\]
When $\lambda$ is plugged into $\mathcal{D}^{(r)}p(x)$, each summand in the formula is zero except for $a_r \binom{r}{r} (x - \lambda)^{r - r} = a_r$.  
\end{proof}

This view of Taylor's expansion leads to the following identities.

\begin{corollary}
\label{cor:hasseshift}
Let $p(x)$ be a polynomial in $\mathbb{F}[x]$.  Then 
\[
    \mathcal{D}^{(r)}(x - \lambda)^m p(x) \big|_{x = \lambda} = 
    \begin{cases}
        \mathcal{D}^{(r - m)}p(\lambda) & \text{ if } r \geq m, \\
        0 & \text{ otherwise.}
    \end{cases}
\]

\end{corollary}

In the subsequent discussion, $\mathcal{D}^{(r)}$ will always be treated as the Hasse 
derivative so that results obtained are valid
for arbitrary fields. The simple observation in the corollary turns out to be a very useful tool for proving
results on the confluent Vandermonde matrix as we shall see.

One may see \cite{Goldschmidt03} for  more  information about the Hasse derivative.

\section{LU factorization}

In the subsequent discussion, we shall always 
let $\Lambda = \{\lambda_1^{(m_1)}, \ldots, \lambda_q^{(m_q)}\}$ 
for $q$ distinct elements in $\lambda_1, \dots, \lambda_q$ in a field ${\mathbb F}$
with multiplicities $m_1, \dots, m_q$ such that
$d = m_1 + \cdots + m_q$.  
We will find an LU factorization of the confluent Vandermonde matrix 
$V_\Lambda$ in the following.
To achieve our goal, consider the $d+1$ polynomials:
$$p_0(x) = 1,\  p_1(x) = (x - \lambda_1), \ \ldots,  
\ p_{m_1}(x) = (x - \lambda_1)^{m_1}, $$
$$p_{m_1+1} = (x-\lambda_1)^{m_1}(x - \lambda_2), \ \dots, \ p_{d}(x) 
= (x-\lambda_1)^{m_1} \cdots (x-\lambda_q)^{m_q}.$$ 
Thus, $p_i(x)$ is a degree $i$ polynomial with leading coefficient equal to 1.
Here we  only use $p_0(x), \dots, p_{d-1}(x)$, and consider 
\begin{equation}
\label{eq:lu}
    V_\Lambda
    \begin{bmatrix}
        | & ~ & | \\
        [p_0(x)] & \cdots & [p_{d-1}(x)] \\
        | & ~ & | \\
    \end{bmatrix} = 
    \begin{bmatrix}
        | & ~ & | \\
        \eval_\Lambda(p_0(x)) & \cdots & \eval_\Lambda(p_{d-1}(x)) \\
        | & ~ & | \\
    \end{bmatrix}.
\end{equation}
Let $U^{-1} = \Big[ [p_0(x)] \ \cdots \ [p_{d-1}(x)] \Big]$ and 
$L =  \Big[ \eval_\Lambda(p_0(x)) \ \cdots \ \eval_\Lambda(p_{d-1}(x))\Big]$.
Note that $U^{-1}$ is an upper triangular matrix with $1$'s on the diagonal. 
Thus, we can write it as the inverse of some matrix $U$, where $U = (U^{-1})^{-1}$ is an upper triangular matrix.  
Moreover, by \cref{cor:hasseshift}, $L$ is a lower triangular matrix with the diagonal entries $1$ for $m_1$ 
times, $(\lambda_2 - \lambda_1)^{m_1}$ for $m_2$ times, $(\lambda_3 - \lambda_1)^{m_1}(\lambda_3 - \lambda_2)^{m_2}$ 
for $m_3$ times, and so on.

This approach was used by Rushanan\cite{Rushanan89} to develop the LU factorization of $V_\Lambda$ when all elements in $\Lambda$ are distinct and $V_\Lambda$ is the classical Vandermonde matrix.  As we have seen, the same approach leads to the similar consequences for the confluent Vandermonde matrix.  Therefore, we have the following.

\begin{theorem}
\cref{eq:lu}, written as $V_\Lambda U^{-1} = L$,  leads to the LU factorization $V_\Lambda = LU$.
\end{theorem}

This leads to a generalization of \cref{eq:det}.  The next theorem was mentioned in \cite{AT32} and proved in \cite{Aitken39}.  The proof in \cite{Aitken39} requires arithmetic on homogeneous symmetric functions and taking limits when $\mathbb{F} = \mathbb{R}$, while using the LU factorization allows us to observe the formula directly for $V_A$ over arbitrary $\mathbb{F}$.

\begin{theorem}
\label{thm:det}
For any $\Lambda = \{\lambda_1^{(m_1)}, \ldots, \lambda_q^{(m_q)}\}$, $\det(V_\Lambda) = \prod_{\substack{i,j\\i<j}} (\lambda_j - \lambda_i)^{m_im_j}$.
Consequently, $V_\Lambda$ is always invertible.
\end{theorem}
\begin{proof}
This is immediate from the fact that $V_\Lambda U^{-1} = L$ and that the determinant of an upper triangular matrix (or a lower triangular matrix) is the product of its diagonal entries.
\end{proof}

Using \cref{basic-example} with $(\lambda_1,\lambda_2,\lambda_3) = (1,3,6)$, we have the following numerical example. 

\begin{example} 
Let $\Lambda = \{1^{(1)},3^{(2)},6^{(2)}\}$ and $\mathbb{F}$ a field where $1,3,6$ are distinct, or equivalently, $\fchar(\mathbb{F}) \neq 2,3,5$.  Then 
\[
    p_0(x) = 1,\ p_1(x) = (x - 1),\ p_2(x) = (x - 1)(x - 3), 
\]
\[
    p_3(x) = (x - 1)(x - 3)^2,\ p_4(x) = (x - 1)(x - 3)^2(x - 6). 
\]
Also,  
\[
    V_\Lambda = \begin{bmatrix}
        1 & 1 & 1 & 1 & 1 \\
        1 & 3 & 9 & 27 & 81 \\
        0 & 1 & 6 & 27 & 108 \\
        1 & 6 & 36 & 216 & 1296 \\
        0 & 1 & 12 & 108 & 864
    \end{bmatrix} \text{ and }
    U^{-1} = \begin{bmatrix}
        1 & -1 & 3 & -9 & 54 \\
        0 & 1 & -4 & 15 & -99 \\
        0 & 0 & 1 & -7 & 57 \\
        0 & 0 & 0 & 1 & -13 \\
        0 & 0 & 0 & 0 & 1
    \end{bmatrix}.
\]
Thus, 
\[
    V_\Lambda U^{-1} = \begin{bmatrix}
        p_0(1) & \cdots & p_4(1) \\
        p_0(3) & \cdots & p_4(3) \\
        \mathcal{D}^{(1)}p_0(3) & \cdots & \mathcal{D}^{(1)}p_4(3) \\
        p_0(6) & \cdots & p_4(6) \\
        \mathcal{D}^{(1)}p_0(6) & \cdots & \mathcal{D}^{(1)}p_4(6) \\
    \end{bmatrix} =
    \begin{bmatrix}
        1 & 0 & 0 & 0 & 0 \\
        1 & 2 & 0 & 0 & 0 \\
        0 & 1 & 2 & 0 & 0 \\
        1 & 5 & 15 & 45 & 0 \\
        0 & 1 & 8 & 39 & 45
    \end{bmatrix} = L.
\]
Since $\det(U^{-1}) = 1$, it follows that $\det(V_\Lambda) = \det(L) 
= (3 - 1)^2 (6 - 1)^2 (6 - 3)^4$.  
Since $1,3,6$ are distinct in $\mathbb{F}$, we have $\det(L) \neq 0$.  
\end{example}

\begin{remark}
The Chinese remainder theorem states that given the residues of a polynomial modulo $p_i(x)$ for some coprime family of polynomials $\{p_1(x), \ldots, p_q(x)\}$, the polynomial always exists and is uniquely determined, modulo $\prod_{i=1}^q p_i(x)$; see, e.g., \cite{Shoup08}.  When $p_i(x) = (x - \lambda_i)^{m_i}$ for some multiset $\Lambda = \{\lambda_1^{(m_1)}, \ldots, \lambda_q^{(m_q)}\}$, the Chinese remainder theorem is equivalent to the invertibility of $V_\Lambda$.
\end{remark}

\section{Partial fraction decomposition}

Partial fraction decomposition is an important tool for many areas, e.g., calculus, complex analysis, and more.  However, when the denominator has repeated roots, the computation of the partial fraction decomposition is not so straightforward; see, e.g., \cite{MS83}.  

For example, let $\Lambda = \{1^{(1)},3^{(2)},6^{(2)}\}$.  Given a polynomial $p(x)$ of degree at most $4$, how to find the coefficients $a,b,c,d,e$ such that  
\begin{equation}
\label{eq:partial}
    \frac{p(x)}{(x-1)(x-3)^2(x-6)^2} = \frac{a}{(x-1)} + \frac{b + c(x-3)}{(x-3)^2} + \frac{d + e(x-6)}{(x-6)^2}
\end{equation}
holds?

For the real and complex field, when a root $\lambda$ is simple, one may multiply on both sides by $(x-\lambda)$ and take the limit $x\rightarrow \lambda$ to find the corresponding coefficients.  For example, 
\[
    a = p(1) \lim_{x\rightarrow 1} \frac{(x - 1)}{(x-1)(x-3)^2(x-6)^2}.
\]
However, the same idea does not extend to the case for repeated roots and polynomials over arbitrary fields.  Here we demonstrate an efficient approach using the evaluation $\eval(p(x))$ and the Vandermonde matrix $V_\Lambda$, which is similar to the methods in \cite{MS83}.   

We first formulate the problem.  Let $\Lambda = \{\lambda_1^{(m_1)}, \ldots, \lambda_q^{(m_q)}\}$ with $d = m_1 + \cdots + m_q$.  Define 
\[
    \hat{h}_{j,m}(x) = (x - \lambda_j)^m \prod_{\substack{k = 1, \ldots, q\\ k\neq j}}
    (x - \lambda_k)^{m_k}
\]
for $j = 1, \ldots, q$ and $m = 0, \ldots, m_j - 1$. Note that there are $d$ polynomials. 
When we write $\hat{h}_{1,0}(x)$, $\ldots$, $\hat{h}_{q,m_q-1}(x)$, we mean $\hat{h}_{1,0}(x)$, $\ldots$, $\hat{h}_{1,m_1-1}(x)$, 
$\ldots$, $\hat{h}_{q,0}(x)$, $\ldots$, $\hat{h}_{q,m_q-1}(x)$ in order.  By writing \cref{eq:partial} in a common denominator, we observe that finding the partial fraction decomposition is the same as writing $p(x)$ as a linear combination of $\hat{\mathcal{H}}_\Lambda 
= \{\hat{h}_{1,0}(x), \ldots, \hat{h}_{q,m_q-1}(x)\}$.  Let 
\[
    \hat{H}_\Lambda = \begin{bmatrix}
        | & ~ & | \\
        [\hat{h}_{1,0}] & \cdots & [\hat{h}_{q,m_q-1}] \\
        | & ~ & | \\
    \end{bmatrix}.
\]
Then coefficients of the partial fraction decomposition can be found by $\hat{H}_\Lambda^{-1}[p(x)]$.  

\begin{example}
\label{ex:parbyhh}
Let $\Lambda = \{1^{(1)},3^{(2)},6^{(2)}\}$ and $\mathbb{F}$ a field where $1,3,6$ are distinct.  Then  
\[
    \begin{aligned}
        \hat{h}_{1,0}(x) &= (x - 1)^0(x - 3)^2(x - 6)^2, \\
        \hat{h}_{2,0}(x) &= (x - 1)(x - 3)^0(x - 6)^2, \\
        \hat{h}_{2,1}(x) &= (x - 1)(x - 3)^1(x - 6)^2, \\
        \hat{h}_{3,0}(x) &= (x - 1)(x - 3)^2(x - 6)^0, \\
        \hat{h}_{3,1}(x) &= (x - 1)(x - 3)^2(x - 6)^1. \\
    \end{aligned}
\]
By expanding all polynomials, we have 
\[
    \hat{H}_\Lambda = \begin{bmatrix}
        324 & -36 & 108 & -9 & 54 \\
        -324 & 48 & -180 & 15 & -99 \\
        117 & -13 & 87 & -7 & 57 \\
        -18 & 1 & -16 & 1 & -13 \\
        1 & 0 & 1 & 0 & 1
    \end{bmatrix}.
\]
Thus, solving the coefficients in \cref{eq:partial} is equivalent to solving 
\[
    \hat{H}_\Lambda 
    \begin{bmatrix} a \\ b \\ c \\ d \\ e \end{bmatrix}
    = 
    [p(x)].
\]
However, it is time consuming to find the inverse of $\hat{H}_\Lambda$.
\end{example}

The confluent Vandermonde matrix $V_\Lambda$ makes the computation easier.  Suppose we want to solve $\hat{H}_\Lambda \bx = [p(x)]$ for $\bx$.  By multiplying $V_\Lambda$ to the both sides, we have $(V_\Lambda \hat{H}_\Lambda) \bx = \eval(p(x))$.  We will see that $T_\Lambda = V_\Lambda \hat{H}_\Lambda$ is invertible and has nice structure, so that the answer can be readily determined by $\bx = T_\Lambda^{-1} \eval(p(x))$.  

Since $V_\Lambda$ transforms a polynomial into its evaluations, we examine the structure of $V_\Lambda \hat{H}_\Lambda$ by taking the evaluations of $\hat{h}_{j,m}(x)$.  We observe that
\[
    \mathcal{D}^{(r)}\hat{h}_{j,m}(\lambda_i) = 
    \begin{cases}
        0 & \text{if } i\neq j \text{ and } r < m_i, \\
        0 & \text{if } i = j \text{ and } r < m, \\
        \star & \text{if } i = j \text{ and } r = m, \\
        ? & \text{if } i = j \text{ and } r > m. \\
    \end{cases}
\]
Here $\star$ stands for a nonzero number and $?$ stands for any value in the field.  To be more precise, by \cref{cor:hasseshift} we have
\[
    \mathcal{D}^{(m)}\hat{h}_{j,m}(\lambda_j) =
    \mathcal{D}^{(0)} \prod_{\substack{k = 1, \ldots, q\\ k\neq j}} (x - \lambda_k)^{m_k} \Big|_{x = \lambda_j} = 
    \prod_{\substack{k = 1, \ldots, q\\ k\neq j}}
    (\lambda_j - \lambda_k)^{m_k} \neq 0
\]
if $\lambda_i - \lambda_j \ne 0$ whenever $i\ne j$.  Moreover, \cref{cor:hasseshift} also guarantees $\mathcal{D}^{(r)}\hat{h}_{j,m}(\lambda_j) = \mathcal{D}^{(r-1)}\hat{h}_{j,m-1}(\lambda_j)$ if $r \geq 1$ and $m \geq 1$.  \cref{rem:Tstruc} summarizes all these properties.  

\begin{remark}
\label{rem:Tstruc}
For any multiset $\Lambda$, the matrix $T_\Lambda = V_\Lambda \hat{H}_\Lambda$ is a lower triangular matrix with nonzero diagonal entries, so $T_\Lambda$, $V_\Lambda$, and $\hat{H}_\Lambda$ are all invertible, which shows the existence and the uniqueness of the partial fraction decomposition.  Moreover, $T_\Lambda$ is the direct sum of $q$ Toeplitz matrices, where $q$ is the number of distinct values in $\Lambda$.
\end{remark}

Continuing \cref{ex:parbyhh}, we solve the partial fraction decomposition through $T_\Lambda = V_\Lambda \hat{H}_\Lambda$.  

\begin{example}
\label{ex:parbyt}
We first compute $T = V_\Lambda \hat{H}_\Lambda$ through evaluations.  By direct computation, $\mathcal{D}^{(0)}\hat{h}_{1,0}(1) = 100$. 
Also, we may compute that 
\[
    \begin{aligned}
    (x - 1)(x - 6)^2 &= (x - 3 + 2)(x - 3 - 3)^2 \\
    & = 18 - 3(x - 3) - 4(x - 3)^2 + (x - 3)^3,
    \end{aligned}
\]
so $\mathcal{D}^{(0)}\hat{h}_{2,0}(3) = 18$, $\mathcal{D}^{(1)}\hat{h}_{2,0}(3) = -3$, and $\mathcal{D}^{(0)}\hat{h}_{2,1}(3) = 0$, $\mathcal{D}^{(1)}\hat{h}_{2,1}(3) = 18$.  Other evaluations can be done in similar manners.  Therefore, 
\[
    T = V_\Lambda \hat{H}_\Lambda = \begin{bmatrix}
        \hat{h}_{1,0}(1) & \cdots & \hat{h}_{3,1}(1) \\
        \hat{h}_{1,0}(3) & \cdots & \hat{h}_{3,1}(3) \\
        \mathcal{D}^{(1)}\hat{h}_{1,0}(3) & \cdots & \mathcal{D}^{(1)}\hat{h}_{3,1}(3) \\
        \hat{h}_{1,0}(6) & \cdots & \hat{h}_{3,1}(6) \\
        \mathcal{D}^{(1)}\hat{h}_{1,0}(6) & \cdots & \mathcal{D}^{(1)}\hat{h}_{3,1}(6) \\
    \end{bmatrix} =
    \begin{bmatrix}
        100 & 0 & 0 & 0 & 0 \\
        0 & 18 & 0 & 0 & 0 \\
        0 & -3 & 18 & 0 & 0 \\
        0 & 0 & 0 & 45 & 0 \\
        0 & 0 & 0 & 39 & 45
    \end{bmatrix}.
\]
By the structure of $T_\Lambda$, its inverse can be easily found as 
\[
    T^{-1}_\Lambda = \begin{bmatrix}
        \frac{1}{100} & 0 & 0 & 0 & 0 \\
        0 & \frac{1}{18} & 0 & 0 & 0 \\
        0 & \frac{1}{108} & \frac{1}{18} & 0 & 0 \\
        0 & 0 & 0 & \frac{1}{45} & 0 \\
        0 & 0 & 0 & -\frac{13}{675} & \frac{1}{45}
    \end{bmatrix}.
\]
Note that $\fchar(\mathbb{F}) \neq 2,3,5$, the entries in $T_\Lambda^{-1}$ are well defined.
With $\eval(p(x))$ computed, the coefficients of partial fraction decomposition of $p(x)$ are recorded in $T^{-1}_\Lambda \eval(p(x))$.  
\end{example}

\begin{theorem}
\label{thm:parbyt}
Let $\Lambda = \{\lambda_1^{(m_1)}, \dots, 
\lambda_q^{(m_q)}\}$ with $m_1 + \cdots + m_q = d$.  For any given polynomial $p(x)$ of degree at most $d-1$, $\hat{H}^{-1}_\Lambda [p(x)] = T^{-1}_\Lambda \eval(p(x))$ records the coefficients of $f(x)$ as a linear combination of $\hat{\mathcal{H}}_\Lambda$, which gives the partial fraction decomposition of $f(x)/\prod_{i=1}^q (x - \lambda_i)^{m_i}$.
\end{theorem}

\begin{figure}[h]
    \centering
    \begin{tikzpicture}
        \node[draw=none, rectangle] (1) at (-3,0) {$[p(x)]$};
        \node[draw=none, rectangle] (2) at (3,-1) {$\eval(p(x))$};
        \node[draw=none, rectangle, align=center] (3) at (3,1) {partial fraction\\decomposition};
        \draw[-triangle 45] (1.east) --node[midway,below,draw=none,rectangle] {$V_\Lambda$} (2.west);
        \draw[-triangle 45] (3.west) --node[midway,above,draw=none,rectangle] {$\hat{H}_\Lambda$} (1.east);
        \draw[-triangle 45] (3.south) --node[midway,right,draw=none,rectangle] {$T_\Lambda$} (2.north);
    \end{tikzpicture}
    \caption{Relations between $[p(x)]$, $\eval(p(x))$, and the partial fraction decomposition.}
    \label{fig:relations}
\end{figure}

In conclusion, \cref{fig:relations} illustrates the relations between $[p(x)]$, $\eval(p(x))$, and the partial fraction decomposition.  With $\Lambda$ given, each of $V_\Lambda$, $\hat{H}_\Lambda$, and $T_\Lambda$ can be computed directly.  By \cref{rem:Tstruc}, it is not too hard to find $T_\Lambda^{-1}$.  Thus, \cref{thm:parbyt} states that $\hat{H}_\Lambda^{-1} = T_\Lambda^{-1} V_\Lambda$ provides a more efficient way to compute the partial fraction decomposition.

\section{Hermite interpolation polynomials}

Given the evaluations of a polynomial at different points, the Lagrange interpolation polynomials provide a powerful tool in reconstructing the polynomial.  What happens if the evaluations have derivatives involved?

Let $\Lambda = \{\lambda_1^{(m_1)}, \ldots, \lambda_q^{(m_q)}\}$ with $d = m_1 + \cdots + m_q$.  Let $H_\Lambda = V_\Lambda^{-1}$.  Then we may find polynomials $\mathcal{H}_\Lambda = \{h_{1,0}(x), \ldots, h_{q,m_q-1}(x)\}$ such that 
\[
    H_\Lambda = \begin{bmatrix}
        | & ~ & | \\
        [h_{1,0}(x)] & \cdots & [h_{q,m_q-1}(x)] \\
        | & ~ & | \\
    \end{bmatrix}.
\]
The set of polynomials $\mathcal{H}_\Lambda$ is called the \emph{Hermite interpolation polynomials} of $\Lambda$.  Given any polynomial $p(x)$ of degree at most $d-1$, one may recover the polynomial from the information $\eval_\Lambda(p(x))$ by  
\[
    \relax[p(x)] = V_\Lambda^{-1} \eval_\Lambda(p(x)) = H_\Lambda \eval_\Lambda(p(x))
\]
and write $p(x)$ as a linear combination of polynomials in $\mathcal{H}_\Lambda$.

In general, computing $V_\Lambda^{-1}$ and formulating the Hermite interpolation polynomials are not straightforward; see, e.g., Respondek~\cite{Respondek24} for a survey.  With \cref{fig:relations} in mind, finding the Hermite interpolation polynomials becomes easier with $V_\Lambda^{-1} = \hat{H}_\Lambda^{-1} T_\Lambda^{-1}$.

\begin{example} Let $\Lambda = \{1^{(1)},3^{(2)},6^{(2)}\}$ and $\mathbb{F}$ a field where $1,3,6$ are distinct.
Suppose $p(x)\in {\mathbb F}[x]$ is a polynomial of degree at most $4$,
and $p(1) = c_1$, $p(3) = c_2$, $\mathcal{D}^{(1)}(p(3)) = c_3$, $p(6) = c_4$, $\mathcal{D}^{(1)}(p(6)) = c_5$ are given.  
 Then
\[
    V_\Lambda [p(x)] = \eval_\Lambda(p(x)) = \begin{bmatrix} c_1 \\ c_2 \\ c_3 \\ c_4 \\ c_5 \end{bmatrix}.
\]
With $H_\Lambda = V_\Lambda^{-1}$, we have $H_\Lambda \eval_\Lambda(p(x)) = [p(x)]$.  Let $T_\Lambda^{-1}$ and $\hat{h}_{1,0}, \ldots, \hat{h}_{3,1}$ be as in \cref{ex:parbyt}.  Since $H_\Lambda = V_\Lambda^{-1} = \hat{H}_\Lambda T^{-1}_\Lambda$, the Hermite interpolation polynomials are 
\[
    \begin{aligned}
        h_{1,0}(x) &= \frac{1}{100}\hat{h}_{1,0}(x) = \left(\frac{1}{100}\right)(x - 3)^2(x - 6)^2, \\
        h_{2,0}(x) &= \frac{1}{18}\hat{h}_{2,0}(x) + \frac{1}{108}h_{2,1}(x) = (x - 1)\left(\frac{1}{18} + \frac{1}{108}(x-3)\right)(x - 6)^2, \\
        h_{2,1}(x) &= \frac{1}{18}\hat{h}_{2,1}(x) = (x - 1)\left(\frac{1}{18}(x - 3)\right)(x - 6)^2, \\
        h_{3,0}(x) &= \frac{1}{45}\hat{h}_{3,0}(x) - \frac{13}{675}\hat{h}_{3,1}(x) = (x - 1)(x - 3)^2\left(\frac{1}{45} - \frac{13}{675}(x - 6)\right), \\
        h_{3,1}(x) &= \frac{1}{45}\hat{h}_{3,1}(x) = (x - 1)(x - 3)^2\left(\frac{1}{45}(x - 6)\right). \\
    \end{aligned}
\]
For each of these polynomials, its evaluations given by $\Lambda$ are all $0$ except for a unique $1$.  Therefore, for any given evaluations $c_1, \ldots, c_5$, we have the Hermite interpolation
\[
    p(x) = c_1h_{1,0}(x) + c_2h_{2,0}(x) + c_3h_{2,1}(x) + c_4h_{3,0}(x) + c_5h_{3,1}(x).
\]
\end{example}

Note that computing $T^{-1}_\Lambda$ is much easier than computing $V_\Lambda^{-1}$.  
In fact, the author of \cite{HP02} used similar ideas to compute the inverse of a 
confluent Vandermonde matrix, and a faster algorithm was later proposed by Respondek~\cite{Respondek21}.  The Hermite interpolation is also used to calculate the functions of companion matrices 
\cite{Eller87, MacDuffee50}.

\section{Companion matrix}

Let 
\[
    p(x) = c_0 + c_1x + \cdots + c_{d}x^{d} \text{ with } c_d = 1
\]
be a polynomial of degree $d$.  The \emph{companion matrix} of $p(x)$ is defined as 
\[
    C = \begin{bmatrix}
        0 & \cdots & 0 & -c_0 \\
        1 & \ddots & \vdots & -c_{1} \\
        ~ & \ddots & 0 & \vdots \\
        0 & ~ & 1 & -c_{d-1}
    \end{bmatrix}
\]
When 
\[
    p(x) = \prod_{\lambda\in\Lambda}(x - \lambda).
\]
is the monic polynomial with roots in $\Lambda$, we write $C_\Lambda$ as the companion matrix of $p(x)$.
Here, we always assume that the roots lie in the field. In particular, the discussion always works for
an algebraically closed field.

Suppose $\Lambda = \{\lambda_1^{(m_1)}, \ldots, \lambda_q^{(m_q)}\}$.  
Let $J_{\lambda,m}$ be the $m\times m$ lower-triangular Jordan block of $\lambda$ and define 
\[
    J_\Lambda = \begin{bmatrix}
        J_{\lambda_1,m_1} & ~ & ~ \\
        ~ & \ddots & ~ \\
        ~ & ~ & J_{\lambda_q,m_q}
    \end{bmatrix}.
\]
Aitken and Turnbull \cite{AT32} stated that 
\[
    V_\Lambda C_\Lambda V_\Lambda^{-1} = J_\Lambda
\]
is the Jordan canonical form of $C_\Lambda$ with exactly $1$ Jordan block for each distinct eigenvalue.  
In particular, when all elements in $\Lambda$ are distinct, $V_\Lambda C_\Lambda V_\lambda^{-1}$ is a diagonal matrix.  

\begin{example} Let $\Lambda = \{\lambda_1, \lambda_2^{(2)}, \lambda_3^{(2)}\}$ and $p(x) = (x-\lambda_1)(x-\lambda_2)^2(x-\lambda_3)^2$
such that $\lambda_i-\lambda_j \ne 0$ whenever $i\ne j$.
Then $V_\Lambda C_f V_\lambda^{-1}$ will be a direct sum of $[\lambda_1], \begin{bmatrix} \lambda_2 & 0 \cr 
1 & \lambda_2\end{bmatrix}$ and $
 \begin{bmatrix} \lambda_3 & 0 \cr 
1 & \lambda_3\end{bmatrix}$.
\end{example}

This Jordan form result is elegant, but no formal proof was given in \cite{AT32}.
With the discussions in the previous sections, 
we can prove the Jordan form result above and gain better insight into it.  
Recall that $\eval_\Lambda(f(x))$ is a vector evaluating the polynomial $f(x)$.  
By definition, we have $\eval_\Lambda(p(x)) = \bzero$.  Moreover, each evaluation has the following properties.  

\begin{proposition}
\label{prop:prodrule}
Let $\Lambda$ be a multiset and $f(x)\in {\mathbb F}[x]$. Then
\begin{enumerate}
    \item $\mathcal{D}^{(0)}(xf(x))\big|_{x = \lambda} = \lambda \mathcal{D}^{(0)}(f(x))\big|_{x = \lambda}$, 
    \item $\mathcal{D}^{(r)}(xf(x))\big|_{x = \lambda} = \mathcal{D}^{(r-1)}(f(x))\big|_{x = \lambda} + \lambda \mathcal{D}^{(r)}(f(x))\big|_{x = \lambda}$ if $r \geq 1$, and   
    \item $\eval_\Lambda(xf(x)) = J_\Lambda \eval_\Lambda(f(x))$.
\end{enumerate}
\end{proposition}
\begin{proof}
The first identity holds since both sides equal to $\lambda f(\lambda)$.  

For the field of real numbers, the second identity could be written as 
\[
    \frac{1}{r!}\cdot \frac{d^r}{dx^r} (xf(x))\Big|_{x = \lambda} = \frac{1}{r!}\left( \binom{r}{1}\frac{d^{r-1}}{dx^{r-1}} (f(x)) + x\cdot \frac{d^r}{dx^r} (xf(x)) \right)\Big|_{x = \lambda}
\]
and is the product rule applied to $xf(x)$.  However, we emphasize that the identity remains valid from the Hasse derivative point of view, and the proof is even more intuitive.  By direct computation and \cref{cor:hasseshift}, we have 
\[
    \begin{aligned}
        \mathcal{D}^{(r)}(xf(x))\big|_{x = \lambda} &= \mathcal{D}^{(r)}((x - \lambda)f(x) + \lambda f(x))\big|_{x = \lambda} \\
        &= \mathcal{D}^{(r)}((x - \lambda)f(x))\big|_{x = \lambda} + \mathcal{D}^{(r)}(\lambda f(x))\big|_{x = \lambda} \\
        &= \mathcal{D}^{(r-1)}(f(x))\big|_{x = \lambda} + \lambda\mathcal{D}^{(r)}(f(x))\big|_{x = \lambda}.
    \end{aligned}
\]

Finally, the third identity is simply a matrix-vector form of the first two identities.  
\end{proof}

Let $p_j(x) = x^j$.  Then the columns of the $d\times d$ identity matrix $I$ are $[p_0(x)], \ldots, [p_{d-1}(x)]$ and 
\[
    C_\Lambda = \begin{bmatrix}
        | & ~ & | & | \\
        [xp_0(x)] & \cdots & [xp_{d-2}(x)] & [xp_{d-1}(x) - p(x)] \\
        | & ~ & | & | \\
    \end{bmatrix}.
\]
With \cref{prop:prodrule} and the fact $\eval_\Lambda(p(x)) = \bzero$, we have 
\[
    V_\Lambda C_\Lambda = 
    J_\Lambda \begin{bmatrix}
        | & ~ & | \\
        \eval_\Lambda(p_0(x)) & ~ & \eval_\Lambda(p_{d-1}(x)) \\
        | & ~ & | \\
    \end{bmatrix} = J_\Lambda V_\Lambda I = J_\Lambda V_\Lambda.
\]
Therefore, $V_\Lambda C_\Lambda V_\Lambda^{-1} = J_\Lambda$ follows.

Moreover, with our discussion in earlier sections, $V_\Lambda^{-1}$ comes from the Lagrange interpolation polynomials when all elements in $\Lambda^{-1}$ are distinct and from the Hermite interpolation polynomials in general.  Let $\mathcal{E}_d = \{p_0(x), \ldots, p_{d-1}(x)\}$ and $\mathcal{H}_\Lambda$ the Hermite interpolation polynomials.  Then both $\mathcal{E}_d$ and $\mathcal{H}_\Lambda$ are bases of the space $\mathcal{P}_{d-1}$ of polynomials of degree at most $d-1$.  Consider the linear transformation $\phi: \mathcal{P}_{d-1} \rightarrow \mathcal{P}_{d-1}$ defined by $\phi(f(x)) = xf(x) \pmod{p(x)}$.  Then $C_\Lambda$ can be viewed as the matrix representation of $\phi$ with respect to $\mathcal{E}_d$, while $J_\Lambda$ is the matrix representation of $\phi$ with respect to $\mathcal{H}_\Lambda$.  Thus, $V_\Lambda$ and $H_\Lambda = V_\Lambda^{-1}$ are the change of basis matrices between $\mathcal{E}_d$ and $\mathcal{H}_\Lambda$.  This provides a comprehensive understanding of the Jordan decomposition of $C_\Lambda$.

The decomposition $V_\Lambda C_\Lambda V_\Lambda^{-1} = J_\Lambda$ was then used for solving differential equations in \cite{Brand68} and yet more applications in \cite{MacDuffee50,Watkins78,EM95}.  

Moreover, Gilbert \cite{Gilbert88} and the references therein considered integer matrices with integer eigenvalues with integer (left and right) generalized eigenvectors.  The paper pointed out that $C_\Lambda$ is one of such matrices when $\Lambda$ contains only two distinct values that are integers with difference $1$.  This can again be seen from our discussion.  When $\Lambda$ contains only one integer value or two distinct integer values with difference $1$, we have $\det(V_\Lambda) = 1$ and $V_\Lambda^{-1}$ is an integer matrix.  In these cases, $V_\Lambda$, $C_\Lambda$, $V_\Lambda^{-1}$, and $J_\Lambda$ are all integer matrices.

\section*{Acknowledgements}
We thank the referee for their careful reading and
helpful suggestions, which improved the quality of the paper.
Chi-Kwong Li is an affiliate member of the Institute for Quantum Computing, University of Waterloo;
 his research was partially supported by the Simons Foundation Grant 851334.
Jephian C.-H. Lin was supported by the National Science and Technology Council of Taiwan (grant no.\ NSTC-112-2628-M-110-003 and grant no.\ NSTC-113-2115-M-110-010-MY3).




\end{document}